\documentclass[10pt]{amsart}
\usepackage{amsmath,amssymb,latexsym,tabu}

\newtheorem{theorem}{Theorem}
\newtheorem{lemma}[theorem]{Lemma}

\begin{document}

\title{de Bruijn arrays for L-fillings }
\author{Lara Pudwell and Rachel Rockey}
\thanks{This project was partially supported by the National Science Foundation (NSF DUE-1068346).}
\address{
	Department of Mathematics \& Computer Science\\
	Valparaiso University\\
	Valparaiso, Indiana 46383, USA
}

\begin{abstract}
We use modular arithmetic to construct a de Bruijn array containing all fillings of an L (a $2 \times 2$ array with the upper right corner removed) with digits chosen from $\{0,\dots,k-1\}$.
\end{abstract}

\maketitle
\markboth{Lara Pudwell and Rachel Rockey}{de Bruijn arrays for Ls}

\section{Introduction}\label{Introduction}

How many sequences of length 3 can be made with the digits 0 and 1?  Exhaustive listing reveals 8 such sequences: 000, 001, 010, 100, 011, 101, 110, and 111.  More generally, introductory combinatorics tells us that there are $k^n$ sequences of length $n$ made up from the digits $\{0, \dots, k-1\}$.  A more challenging question is: is there a sequence of length $k^n$ that contains each of these $k^n$ sequences as a consecutive subsequence?  (Here we consider the beginning and end of the sequence to be glued together.)  Such a sequence is known as a $(k,n)$-\emph{de Bruijn sequence}.  For example, 00010111 is a (2,3)-de Bruijn sequence.  (000, 001, 010, 101, 011, and 111 are easy to spot; 110 and 100 make use of the glued ends of the sequence.)

In fact, such sequences are well-studied and have been used in applications ranging from robotics to developing card tricks.  Diaconis and Graham give a delightful overview of such applications in \cite{DG11}.  Nicolaas Govert de Bruijn \cite{dB46}, for whom such sequences are named, and I.J. Good \cite{G46} independently developed a special type of graph in the 1940s that is instrumental in the proof of Theorem \ref{deBseq}.

\begin{theorem}\label{deBseq}
$(k,n)$-de Bruijn sequences exist for every $k, n \geq 2$, and there are $\frac{k!^{k^{n-1}}}{k^n}$ of them.
\end{theorem}

Much more recently, mathematicians have analyzed a 2-dimensional generalization of these sequences known as de Bruijn tori.  Given $k \geq 2$ and  $m, n \in \mathbb{Z}^+$ a $(k,m,n)$-\emph{de Bruijn torus} is an array that contains each of the $k^{mn}$ fillings of an $m \times n$ array with entries from $\{0,\dots, k-1\}$ exactly once.  (Here, the left side of the array is glued to the right side, and the top of the array is glued to the bottom.)  Much is still unknown about general de Bruijn tori.  However, when $m=1$, a de Bruijn torus is merely a de Bruijn sequence.  Further, Jackson, Stevens, and Hurlbert \cite{JSH09} showed that $(k,m,n)$-de Bruijn tori always exist when $m=n$.  An example of a $(2,2,2)$-de Bruijn torus is shown in Figure \ref{deBtorus}.  Here, the all 1s filling is easiest to spot.  The all 0s filling is trickiest: it uses the 0s in the four corners of the array, which (thanks to the gluing of top to bottom and the gluing of left to right) are all adjacent.

\begin{figure}
\begin{center}
\begin{tabular}{|c|c|c|c|}
\hline
0&0&1&0\\
\hline
1&1&1&0\\
\hline
0&1&1&1\\
\hline
0&1&0&0\\
\hline
\end{tabular}
\end{center}
\caption{A (2,2,2)-de Bruijn torus}
\label{deBtorus}
\end{figure}

In this note, we consider an intermediate problem.  Rather than compactly arranging all possible sequences or all possible fillings of a rectangular array, we consider fillings of an L-shape, that is, a $2 \times 2$ array with the upper right corner removed.  Since the L-shape has 3 entries, there are $k^3$ fillings of the L with digits from $\{0,\dots, k-1\}$.  For example, the $2^3=8$ fillings of an L with a binary alphabet are shown in Figure \ref{binaryLs}.  In general, we wish to find a $k \times k^2$ array (with the left side glued to the right side an the top glued to the bottom) that contains each filling of the L exactly once.  To distinguish from the tori in the previous paragraph, an array that contains all fillings of an L with digits from $\{0,\dots , k-1\}$ is a $k$-\emph{de Bruijn L-array}.  An example of a 2-de Bruijn L-array is shown in Figure \ref{binarydeB}.

\begin{figure}
\begin{center}
\begin{tabular}{cccc}

\begin{tabular}{|c|c|}
\cline{1-1}
\multicolumn{1}{|c|}{0}\\
\hline
0&0\\
\hline
\end{tabular}&

\begin{tabular}{|c|c|}
\cline{1-1}
\multicolumn{1}{|c|}{0}\\
\hline
0&1\\
\hline
\end{tabular}&

\begin{tabular}{|c|c|}
\cline{1-1}
\multicolumn{1}{|c|}{0}\\
\hline
1&0\\
\hline
\end{tabular}&

\begin{tabular}{|c|c|}
\cline{1-1}
\multicolumn{1}{|c|}{0}\\
\hline
1&1\\
\hline
\end{tabular}\\

\vspace{0.1in}&&&\\

\begin{tabular}{|c|c|}
\cline{1-1}
\multicolumn{1}{|c|}{1}\\
\hline
0&0\\
\hline
\end{tabular}&

\begin{tabular}{|c|c|}
\cline{1-1}
\multicolumn{1}{|c|}{1}\\
\hline
0&1\\
\hline
\end{tabular}&

\begin{tabular}{|c|c|}
\cline{1-1}
\multicolumn{1}{|c|}{1}\\
\hline
1&0\\
\hline
\end{tabular}&

\begin{tabular}{|c|c|}
\cline{1-1}
\multicolumn{1}{|c|}{1}\\
\hline
1&1\\
\hline
\end{tabular}
\end{tabular}
\end{center}
\caption{All possible binary L fillings}
\label{binaryLs}
\end{figure}

\begin{figure}
\begin{center}
\begin{tabular}{|c|c|c|c|}
\hline
0&0&1&0\\
\hline
0&1&1&1\\
\hline
\end{tabular}
\end{center}
\caption{A 2-de Bruijn L-array}
\label{binarydeB}
\end{figure}

It turns out that such arrays always exist.  The rest of this paper is aimed at a proof of Theorem \ref{Larray}.

\begin{theorem}\label{Larray}
For $k \geq 2$, there exists a $k$-de Bruijn L-array.
\end{theorem}

\section{Coordinates}

To easily refer to the entries of our de Bruijn L-arrays, we will index each entry in the array with 3 coordinates: row, column, and square.  

We index rows starting with 0.  Since a $k$-de Bruijn L-array has $k$ rows, a row number $r$ must have $0 \leq r \leq k-1$.

We also index columns starting with 0.  A $k$-de Bruijn L-array has $k^2$ columns, and we partition the columns into $k$ squares.  Columns $0, 1, \dots, k-1$ make up square 0; columns $k, \dots, 2k-1$ make up square 1; in general, columns $sk, \dots, (s+1)k-1$ make up square $s$.  Like the row number, an entry's square number $s$ must have $0 \leq s \leq k-1$.

Finally, we label an entry's column not by its column number within the entire array but by its column number within its square, starting with 0.  Equivalently, if we wish to talk about an entry in the $j$th column of the array, write $j=sk+c$ where $0 \leq s, c \leq k-1$ to get the square number $s$ and the column number $c$.

As an example, consider the $3 \times 9$ array in Figure \ref{arraypos}.  $a$ is in row 0, column 0, square 0.  $b$ is in row 0, column 0, square 1.  $c$ is in row 1, column 0, square 2.  $d$ is in row 2, column 2, square 2.  This coordinate system uniquely identifies each of the $k^3$ entries in a $k \times k^2$ array with a 3-tuple in $\{0, \dots, k-1\}^3$.

\begin{figure}
\begin{center}
\begin{tabu}{|[2pt]c|c|c|[2pt]c|c|c|[2pt]c|c|c|[2pt]}
\tabucline[2pt]{-}
a&\phantom{x}&\phantom{x}&b&\phantom{x}&\phantom{x}&\phantom{x}&\phantom{x}&\phantom{x}\\
\hline
&&&&&&c&&\\
\hline
&&&&&&&&d\\
\tabucline[2pt]{-}
\end{tabu}
\end{center}
\caption{A $3 \times 9$ array}
\label{arraypos}
\end{figure}

It turns out that the array with $((s+rc) \bmod k)$ in row $r$, column $c$, square $s$ is a $k$-de Bruijn L-array.  For example, the 2-de Bruijn L-array, the 3-de Bruijn L-array, and the 4-de Bruijn L-array produced by this formula are shown in Figure \ref{twoegs}.

\begin{figure}
\begin{center}
\begin{tabu}{|[2pt]c|c|[2pt]c|c|[2pt]}
\tabucline[2pt]{-}
0&0&1&1\\
\hline
0&1&1&0\\
\tabucline[2pt]{-}
\end{tabu}

\vspace{0.2in}

\begin{tabu}{|[2pt]c|c|c|[2pt]c|c|c|[2pt]c|c|c|[2pt]}
\tabucline[2pt]{-}
0&0&0&1&1&1&2&2&2\\
\hline
0&1&2&1&2&0&2&0&1\\
\hline
0&2&1&1&0&2&2&1&0\\
\tabucline[2pt]{-}
\end{tabu}

\vspace{0.2in}

\begin{tabu}{|[2pt]c|c|c|c|[2pt]c|c|c|c|[2pt]c|c|c|c|[2pt]c|c|c|c|[2pt]}
\tabucline[2pt]{-}
0&0&0&0&1&1&1&1&2&2&2&2&3&3&3&3\\
\hline
0&1&2&3&1&2&3&0&2&3&0&1&3&0&1&2\\
\hline
0&2&0&2&1&3&1&3&2&0&2&0&3&1&3&1\\
\hline
0&3&2&1&1&0&3&2&2&1&0&3&3&2&1&0\\
\tabucline[2pt]{-}
\end{tabu}
\end{center}
\caption{A 2-de Bruijn L-array, a 3-de Bruijn L-array, and a 4-de Bruijn L-array}
\label{twoegs}
\end{figure}

\section{Some Number Theory}

For this section, given $k \geq 2$, we consider the array with $((s+rc) \bmod k)$ in row $r$, column $c$, square $s$.   To prove Theorem \ref{Larray}, we need to show that any two L-fillings in different positions in the array are different fillings.  First, we have a lemma.

\begin{lemma}\label{diffcol}
Given two fillings \begin{tabular}{|c|c|}
\cline{1-1}
\multicolumn{1}{|c|}{$a$}\\
\hline
$b$&$d_1$\\
\hline
\end{tabular}
and 
\begin{tabular}{|c|c|}
\cline{1-1}
\multicolumn{1}{|c|}{$a$}\\
\hline
$b$&$d_2$\\
\hline
\end{tabular}, 
 $d_1$ and $d_2$ have the same column number.
\end{lemma}

\begin{proof}
By construction, if $b$ appears immediately below $a$ in the array, then their column number $c$ is given by $c=(b-a) \bmod k$.  This means that if $(b-a) \bmod k = k-1$, then $d_1$ and $d_2$ both have column number $0$.  Otherwise, $d_1$ and $d_2$ both have column number $((b-a) \bmod k) + 1$.
\end{proof}

We continue by analyzing the row numbers of these fillings.

\begin{lemma}\label{diffrow}
Given two fillings \begin{tabular}{|c|c|}
\cline{1-1}
\multicolumn{1}{|c|}{$a$}\\
\hline
$b$&$d_1$\\
\hline
\end{tabular}
and 
\begin{tabular}{|c|c|}
\cline{1-1}
\multicolumn{1}{|c|}{$a$}\\
\hline
$b$&$d_2$\\
\hline
\end{tabular} in different locations of the array, $d_1$ and $d_2$ have different row numbers.
\end{lemma}

\begin{proof}
To compute the digit at row $r$, column $c$, square $s$, we use the formula $(s+rc) \bmod k$. If the $a$ in the first filling is in square $s$, row $r$, and column $c$ while the $a$ in the second filling is in square $\hat{s}$, row $\hat{r}$, and column $c$, we know $a=(s+rc) \bmod k = (\hat{s} + \hat{r}c) \bmod k$. 

If $r= \hat{r}$ we see that $(s+rc) \bmod k = (\hat{s} + \hat{r}c) \bmod k$ and, after subtracting $rc$ from both sides, $s=\hat{s} \bmod k$. This means both fillings have the same square, row, and column numbers, which contradicts these two fillings being in different locations.  It must be the case that $r \not= \hat{r}$.
\end{proof}

Now we are ready to prove Theorem \ref{Larray} by showing that the previously-described array is indeed a 2-de Bruijn L-array. 

\begin{proof}
Suppose we have two fillings, \begin{tabular}{|c|c|}
\cline{1-1}
\multicolumn{1}{|c|}{$a$}\\
\hline
$b$&$d_1$\\
\hline
\end{tabular}
and 
\begin{tabular}{|c|c|}
\cline{1-1}
\multicolumn{1}{|c|}{$a$}\\
\hline
$b$&$d_2$\\
\hline
\end{tabular}, which are in different locations of the array.
We wish to show $d_1 \not= d_2$. By Lemma \ref{diffcol}, we know that the two occurrences of $a$ have the same column number $c$. 

If $c=k-1$, we see by construction that the $a$s in the two fillings must have different square numbers $s$ and $\hat{s}$.  Furthermore, $d_1 = (s+1) \bmod k$ and $d_2 = (\hat{s}+1) \bmod k$.  Therefore $d_1 \neq d_2$ because $s \neq \hat{s}$.

Now suppose $c \neq k-1$, the first filling has its $a$ entry in square $s$ and row $r$, and the second filling has its $a$ entry in square $\hat{s}$ and row $\hat{r}$.  We see:\\
$a=(s+rc) \bmod k = (\hat{s}+\hat{r}c) \bmod k$,\\
$b=(s+(r+1)c) \bmod k = (\hat{s}+(\hat{r}+1)c) \bmod k$,\\
$d_1=(s+(r+1)(c+1)) \bmod k = (s+rc+c+r+1) \bmod k = (a + c + r + 1) \bmod k$,\\
$d_2=(\hat{s}+(\hat{r}+1)(c+1)) \bmod k = (\hat{s} + \hat{r} c+c+ \hat{r}+1) \bmod k = (a+ c + \hat{r} + 1) \bmod k$.

By Lemma \ref{diffcol}, we know that $r \neq \hat{r}$, so $d_1 \neq d_2$.
\end{proof}

\section{For Further Exploration}

Although we have constructed a $k$-de Bruijn L-array for every $k \geq 2$, many interesting questions remain for de Bruijn L-arrays.

For example, when $k=2$, a case analysis shows that (up to rotation) there are precisely two 2-de Bruijn L-arrays; one is shown in Figure \ref{binarydeB} and the other is shown in Figure \ref{twoegs}.  A computer search for other L-arrays when $k=3$ yields dozens more solutions.  Some have evident symmetry while others do not; two examples are given in Figure \ref{nonsymmegs}.  A computer search for other L-arrays when $k\geq 4$ is prohibitive, so less naive approaches are needed.  In particular, for $k >2$ it is unknown how many $k$-de Bruijn L-arrays exist.

\begin{figure}[bth]
\begin{center}
\begin{tabular}{|c|c|c|c|c|c|c|c|c|}
\hline
0&0&0&1&1&1&2&2&2\\
\hline
1&0&0&2&1&1&0&2&2\\
\hline
1&2&1&2&0&2&0&1&0\\
\hline
\end{tabular}

\vspace{0.2in}

\begin{tabular}{|c|c|c|c|c|c|c|c|c|}
\hline
0&1&1&1&0&1&2&2&1\\
\hline
0&0&1&1&2&1&0&0&2\\
\hline
0&2&0&0&2&2&1&2&2\\
\hline
\end{tabular}
\end{center}
\caption{Two 3-de Bruijn L-arrays}
\label{nonsymmegs}
\end{figure}

Other two-dimensional shapes merit further investigation.  One could explore de Bruijn arrays for fillings of staircase shapes or for fillings of the intersection of a longer column with a longer row.  The construction of this paper is particular to the 3-square L discussed above, though, and it is not easily generalized.

\end{document}